\documentclass[11pt]{article}
\usepackage[margin=1in]{geometry} 
\usepackage{url}
\usepackage{graphicx} 
\DeclareGraphicsExtensions{.pdf,.eps,.fig,.pdf_tex}
\usepackage{float}

\usepackage[colorlinks,
             linkcolor=black!50!red,
             citecolor=blue,
             pdftitle={Topological K-Theory for Hilbert Scheme Analogs},
             pdfauthor={Ammar Husain},
             pdfsubject={Finite Groups,K-Theory},
             pdfkeywords={K-Theory,Hilbert schemes,ALE,Finite Groups,partitions,Weyl groups,Platonic Groups},
             pdfproducer={pdfLaTeX},
             pdfpagemode=None,
             bookmarksopen=true
             bookmarksnumbered=true]{hyperref}

\usepackage{amsmath}
\usepackage{amssymb,amsfonts,bbm,mathrsfs,stmaryrd}

\usepackage[amsmath,thmmarks]{ntheorem}
\usepackage{hyperref}

\theoremstyle{change}

\newtheorem{definition}[equation]{Definition}
\newtheorem{thm}[equation]{Theorem}
\newtheorem{theorem}[equation]{Theorem}
\newtheorem{prop}[equation]{Proposition}
\newtheorem{proposition}[equation]{Proposition}
\newtheorem{lemma}[equation]{Lemma}
\newtheorem{cor}[equation]{Corollary}

\theorembodyfont{\upshape}
\theoremsymbol{\ensuremath{\Diamond}}

\theoremstyle{nonumberplain}

\theoremsymbol{\ensuremath{\Box}}
\newtheorem{proof}{Proof}

\qedsymbol{\ensuremath{\Box}}

\numberwithin{equation}{section}

\usepackage{tikz}
\tikzset{dot/.style={circle,draw,fill,inner sep=1pt}}
\usepackage{braids}
\usetikzlibrary{cd}
\usetikzlibrary{decorations.markings}
\usepackage{array}
\newcolumntype{M}[1]{>{\centering\arraybackslash}m{#1}}
\newcolumntype{P}[1]{>{\centering\arraybackslash}p{#1}}

\newcommand\setof[1]{\{ #1 \}}

\newcommand\abs[1]{ \mid #1 \mid }

\usepackage[utf8]{inputenc}

\makeatletter
\newcommand{\tpitchfork}{%
  \vbox{
    \baselineskip\z@skip
    \lineskip-.52ex
    \lineskiplimit\maxdimen
    \m@th
    \ialign{##\crcr\hidewidth\smash{$-$}\hidewidth\crcr$\pitchfork$\crcr}
  }%
}
\makeatother

\usepackage{cleveref}

\def\doAlgSec{0}

\title{Topological K-Theory for Hilbert Scheme Analogs}
\date{}
\author{Ammar Husain}

\begin{document}
\maketitle

\tikzset{->-/.style={decoration={
  markings,
  mark=at position #1 with {\arrow{>}}},postaction={decorate}}}
\tikzset{-<-/.style={decoration={
  markings,
  mark=at position #1 with {\arrow{<}}},postaction={decorate}}}

\begin{abstract}
In geometric representation theory, it is common to compute equivariant $K$ theory of schemes like $Hilb^n ( \mathbb{A}^2 )$ or $Hilb^n (X)$ for an ALE resolution $X \to \mathbb{A}^2 / \Gamma$. If we abandon the algebraic nature and just look at this homotopically we see close relatives of $BS_n$ and $B(\Gamma \wr S_n)$. Therefore we compute the topological K theory of these classifying spaces to fill in a small gap in the literature.
\end{abstract}

\section{Introduction}

In the seminal paper of Atiyah and Segal \cite{AtiyahSegal}, they show that $K^\bullet (BG)$ for a finite group is isomorphic to the representation ring of $G$ completed at the augmentation ideal. Here we compute some examples. In particular at the sequences that come from Platonic groups, Weyl groups and groups associated to Hilbert schemes of du Val singularities. We can construct interesting generating functions in the cases when the groups come in countable families. In some of these cases this can be interpreted as replacing the usual algebraic $K^0$ and genuine equivariance with the much simpler Borel equivariant topological K theory. This can be a very pale shadow of the Platonic ideal.

\section{K Theory of discrete BG}

\begin{thm}[\cite{Luck} L{\"u}ck]

For finite groups $G$\footnote{Parenthesis are used to distinguish p-adic vs cyclic groups.}

\begin{eqnarray*}
K^0 (BG) &=& \mathbb{Z} \times \prod_{p} (\mathbb{Z}_{(p)})^{r(p,G)}\\
K^1 (BG) &=& 0\\
\end{eqnarray*}
where $r(p,G)$ is the number of conjugacy classes $C$ such that $g \in C$ will have order $p^d$ for some $d \geq 1$. Similarly define $\tilde{r}(p,G) = r(p,G) +1$ with $d=0$ also allowed.

\end{thm}

\begin{definition}[Rank Generating Functions]
If there is a sequence of finite groups $G_n$ with $n \in \mathbb{N}$, define two generating functions as

\begin{eqnarray*}
OGF(p,G_\bullet , x) &=& \sum r(p,G_n) x^n\\
\tilde{OGF}(p,G_\bullet , x) &=& \sum \tilde{r}(p,G_n) x^n\\
\tilde{OGF}(p,G_\bullet , x) &=& OGF(p,G_\bullet , x) + \frac{1}{1-x}\\
\end{eqnarray*}

\end{definition}

\section{Platonic Groups \label{Platonic}}

For finite subgroups $G \subset SL(2,\mathbb{C})$, there are few choices. In this section, we describe their $K^0 (BG)$. The cohomologies of these groups are described in \cite{Tomoda}.

\begin{lemma}
For the sequence of cyclic groups $Cyc_n = \mathbb{Z}_{n+1}$

\begin{eqnarray*}
\tilde{OGF}(p,Cyc , x) &=& \sum_{n =0}^\infty p^{\nu_p (n+1)} x^n\\
&=& \sum_{n =0}^\infty \frac{x^n}{\abs{n+1}_p}\\
\end{eqnarray*}

For the sequence of binary cyclic groups $BinCyc_n = \mathbb{Z}_{2(n+1)}$

\begin{eqnarray*}
\tilde{OGF}(p,BinCyc , x) &=& \sum_{n =0}^\infty p^{\nu_p (2(n+1))} x^n\\
&=& \sum_{n =0}^\infty \frac{x^n}{\abs{2(n+1)}_p}\\
\end{eqnarray*}

\end{lemma}

\begin{proof}

\begin{eqnarray*}
\sum_{n} r(p,\mathbb{Z}_{n+1}) x^n &=& \sum_{n} \sum_{r=1} \phi (p^r) \delta_{p^r \mid (n+1)} x^n\\
&=& \sum_{n} \sum_{r=1}^{\nu_p (n+1)} \phi (p^r) x^n\\
\sum_{d \mid n+1} \phi (d) &=& n+1\\
\sum_{n=0}^\infty r(p,G_n) x^n &=& \sum_{n} ( p^{\nu_p (n+1)} - 1) x^n\\
&=& \sum_{n=0}^\infty (\frac{1}{\abs{n+1}_p} - 1) x^n\\
\end{eqnarray*}
where $\nu_p (n)$ is the p-adic valuation that indicates the highest power dividing $n$. $\abs{n}_p = p^{-\nu_p (n)}$ is the p-adic norm.\footnote{For $p=2$, this is \cite{OEIS006519}}.

\end{proof}

\begin{cor}[Product of Cyclic Groups \label{cor:TwoCyclics}]
For $\mathbb{Z}_{n+1} \times \mathbb{Z}_{m+1}$ (which is realized in $GL(2,\mathbb{C})$ instead of $SL(2,\mathbb{C})$) we have

\begin{eqnarray*}
\tilde{OGF}(p,Cyc \times Cyc , x,y) &=& \sum_{n =0}^\infty \sum_{m=0}^{\infty} p^{\nu_p (n+1)} p^{\nu_p (m+1)} x^n y^m\\
&=& \sum_{n=0}^\infty \sum_{m =0}^\infty \frac{x^n y^m}{\abs{n+1}_p \abs{m+1}_p}\\
\end{eqnarray*}

In particular if the sequence of groups $\mathbb{Z}_{n+1}^2$ we get

\begin{eqnarray*}
\tilde{OGF}(p,Cyc \times Cyc , x) &=& \sum_{n=0}^\infty \frac{x^n}{\abs{n+1}_p^2}\\
\end{eqnarray*}
in analogy with the dilogarithm.

\end{cor}

\begin{proposition}[Dihedral]
\begin{eqnarray*}
\tilde{OGF}(2,BinDih_{\bullet+1},x) &=& \sum_{n=0}^\infty (2 + \frac{1}{2 \abs{2n+2}_2}) x^n\\
&=& \frac{2}{1-x} + \frac{1}{2} \tilde{OGF} (p,BinCyc,x)\\
\tilde{OGF}(p,BinDih_{\bullet + 1},x) &=& \sum_{n=0}^\infty 1 + \frac{1}{2}( \frac{1}{\abs{2n+2}_p} - 1) x^n\\
&=& \frac{1}{2} \frac{1}{1-x} + \frac{1}{2} \tilde{OGF} ( p,BinCyc,x)\\
\end{eqnarray*}
\end{proposition}

\begin{proof}
The conjugacy classes of the binary dihedral group are

\begin{itemize}
\setlength\itemsep{-1em}
\item $e$\\
\item $a^n = x^2$ which has order $2$ contributes $1$ to $r(2,BinDih_n)$\\
\item $a^m \simeq a^{2n-m}$ for $m \neq n$. They have same order as $m$ does in binary cyclic group $\mathbb{Z}_{2n}$. Contribute half as much as in $r(p,BinCyc_n)$\\
\item $x$ order $4$. Contribute $1$ to $r(2,BinDih_n)$\\
\item $ax$ order $4$. Contribute $1$ to $r(2,BinDih_n)$\\
\end{itemize}

\begin{eqnarray*}
r(2,BinDih_n) &=& 1+2+ \frac{1}{2}( r(2,\mathbb{Z}_{2n}) -1)\\
r(p \neq 2,BinDih_n)&=&\frac{1}{2} ( \tilde{r}(p,\mathbb{Z}_{2n}) - 1 )\\
\tilde{r}(2,BinDih_n) &=& 2 + \frac{1}{2} \tilde{r}(2,\mathbb{Z}_{2n}) = 2 + \frac{1}{2} 2^{\nu_2 (2n)} = 2 + 2^{\nu_2 (2n)-1}\\
&=& 2 + \frac{1}{2 \abs{2n}_2}\\
\tilde{r}(p,BinDih_n) &=& 1+ \frac{1}{2} ( \tilde{r}(p,\mathbb{Z}_{2n}) - 1 )\\
&=& 1 + \frac{1}{2}(  p^{\nu_p (2n)} - 1)\\
&=& 1 + \frac{1}{2}( \frac{1}{\abs{2n}_p} - 1)\\
\end{eqnarray*}
\end{proof}

The rest are the exceptional types which do not come in sequences. So we just list their $r(p,G)$ for later use.

\begin{lemma}[Exceptional Platonic Groups]

\begin{eqnarray*}
K^0 (BA_4 = BT) &\simeq& \mathbb{Z} \times \mathbb{Z}_{(2)}^1 \times \mathbb{Z}_{(3)}^2\\
K^0 ( BS_4 = BT_d) &\simeq& \mathbb{Z} \times \mathbb{Z}_{(2)}^{3} \times \mathbb{Z}_{(3)}\\
K^0 ( BS_5 ) &\simeq& \mathbb{Z} \times \mathbb{Z}_{(2)}^{3} \times \mathbb{Z}_{(3)} \times \mathbb{Z}_{(5)}\\
K^0 ( BA_5 = BI) &\simeq& \mathbb{Z} \times \mathbb{Z}_{(2)} \times \mathbb{Z}_{(3)} \times \mathbb{Z}_{(5)}^2\\
K^0 (BBinT = BSL(2,3)) &\simeq& \mathbb{Z} \times \mathbb{Z}_{(2)}^2 \times \mathbb{Z}_{(3)}^2\\
K^0 ( BBinI = BSL(2,5)) &\simeq& \mathbb{Z} \times \mathbb{Z}_{(2)}^2 \times \mathbb{Z}_{(3)}^1 \times \mathbb{Z}_{(5)}^2\\
K^0 ( BBinO) &\simeq& \mathbb{Z} \times \mathbb{Z}_{(2)}^5 \times \mathbb{Z}_{(3)}^1\\
\end{eqnarray*}

For $G \times \mathbb{Z}_2$ or $G \times \mathbb{Z}_4$ with one of the $G$ above, simply double/quadruple $\tilde{r}(2,G)$ and leave the others the same. This takes care of all the exceptional finite subgroups of $SO(3) \; O(3) \; Pin_\pm (3)$ and $Spin(3)$.

\end{lemma}

\section{Weyl Groups}

\subsection{Type A}

For $G=S_n$, $\tilde{r}(p,S_n)$ is the number of partitions into powers of $p$. In particular, $\tilde{r}(2,S_n)$ is \cite{OEIS018819} and $r(3,S_n)$ is \cite{OEIS062051}. Note first two trivial groups $S_0 \simeq S_1 \simeq \setof{e}$.

\begin{lemma}
\begin{eqnarray*}
\tilde{OGF} (p,A,x) &\equiv& \sum \tilde{r}(p,n) x^n\\
&=& \prod_{j \geq 0} \frac{1}{1 - x^{p^j}}\\
OGF (p,A,x) &\equiv& \sum r(p,n) x^n\\
&=& \prod_{j \geq 0} \frac{1}{1 - x^{p^j}} - \frac{1}{1-x}\\
&=& \frac{1}{1-x} ( \prod_{j \geq 1} \frac{1}{1 - x^{p^j}} - 1 )\\
g(p,A,x,z) &\equiv& \sum_{j=0}^\infty z^j \sum_{k=1}^{\infty} \frac{x^{kp^j}}{k}\\
g(p,A,x,1) &=& \sum_{j=0}^\infty \sum_{k=1}^{\infty} \frac{x^{kp^j}}{k} = \log \tilde{OGF}(p,A,x)\\
g(p,A,x^p,1) &=& g(p,A,x,1) - \log ( 1 - x)\\
\end{eqnarray*}

Note that without the condition from $p$, this would be related to the logarithm of the Dedekind $\eta$ function.

\end{lemma}

\begin{proof}

\begin{eqnarray*}
g(p,A,x,1) = \log \sum \tilde{r}(p,n) x^n &=& \log \prod_{j \geq 0} \frac{1}{1 - x^{p^j}}\\
&=& \sum_{j=0}^\infty \log \frac{1}{1 - x^{p^j}}\\
&=& - \sum_{j=0}^\infty \log ( 1 - x^{p^j} )\\
&=& - \sum_{j=0}^\infty - \sum_{k=1}^{\infty} \frac{x^{kp^j}}{k}\\
&=& \sum_{j=0}^\infty \sum_{k=1}^{\infty} \frac{x^{kp^j}}{k}\\
\end{eqnarray*}

\begin{eqnarray*}
x \to x^p &\implies& \sum_{j=0}^\infty \sum_{k=1}^{\infty} \frac{x^{kp^j}}{k} \to \sum_{j=0}^\infty \sum_{k=1}^{\infty} \frac{x^{kp^{j+1}}}{k}\\
\sum_{j=0}^\infty z^j \sum_{k=1}^{\infty} \frac{x^{kp^{j+1}}}{k} &=& \sum_{j=0}^\infty z^{j-1} \sum_{k=1}^{\infty} \frac{x^{kp^j}}{k} - z^{-1} \sum_{k=1}^\infty \frac{x^k}{k}\\
g(p,A,x^p , z) &=& \frac{1}{z} ( g(p,A,x,z) - \log (1-x) )\\
z g(p,A,x^p , z) &=& g(p,A,x,z) - \log (1-x)\\
\end{eqnarray*}
\end{proof}

\begin{lemma}
When $x$ approaches the root of unity $x^l=1$ and $(l,p)=1$, there are divergences in $g(p,A,x,z)$. You can begin to see this in \cref{FigA2,FigA3}.
\end{lemma}

\begin{proof}
\begin{eqnarray*}
g(p,A,x,z) &=& \sum_{j=0}^\infty z^j \sum_{k=1}^{l} x^{kp^j} \sum_{m=0}^\infty \frac{1}{k+ml}\\
\sum_{m=0}^\infty \frac{1}{k+ml} &\to& \infty
\end{eqnarray*}

\end{proof}

We may play the usual games that we do with the generating function for all partitions such as finding asymptotics and Mellin transforms.

\begin{lemma}[Asymptotics of $\tilde{r}(p,S_n)$ \cite{Latapy}]
$\log \tilde{r}(p,S_n) \approx \frac{\log^2 n}{2 \log p}$.
This slow growth rate relative to $n!$ indicates $OGF(p,A,x)$ would work well with Borel summation.
\end{lemma}

\begin{lemma}[Mellin Transform]

\begin{eqnarray*}
F(p,A,s) &\equiv& \mathcal{M} ( \log \tilde{OGF}(p,A,e^{-t}))\\
F(p,A,s) &=& \Gamma(s)\sum_{j\ge 0, k \ge 1} \frac{(kp^j)^{-s}}{k} = \Gamma(s)\frac{\zeta(s+1)}{1-p^{-s}}
\end{eqnarray*}

To recover the $\tilde{OGF}$ and therefore the $r(p, S_n)$, undo the Mellin transform and exponentiate.
\end{lemma}

\begin{proof}

Let $j(t) = \log \tilde{OGF} (p,A,e^{-t})$. This obeys a shift equation relating values at $pt$ and $t$.

\begin{eqnarray*}
j (pt) &=& j(t) - \log ( 1 - e^{-t})\\
p^{-s} \mathcal{M}( j(t))(s) &=& \mathcal{M}( j(t))(s) - \mathcal{M} ( \log (1-e^{-t}))\\
(p^{-s} - 1) F(p,A,s) &=& - \mathcal{M} ( \log (1-e^{-t}))\\
\end{eqnarray*}

\end{proof}

\begin{figure}[htb!]
\centering
\includegraphics[scale=.5]{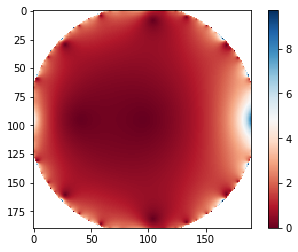}
\caption{Type $A$ $g(p,A,x,1)$ $p=2$ with $j$ and $k$ sums cutoff at $20$. \label{FigA2}}
\end{figure}

\begin{figure}[htb!]
\centering
\includegraphics[scale=.5]{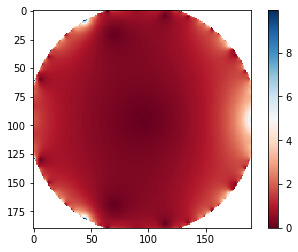}
\caption{Type $A$ $g(p,A,x,1)$ $p=3$ with $j$ and $k$ sums cutoff at $20$. \label{FigA3}}
\end{figure}

\subsection{Type B/C}

\begin{proposition}
The conjugacy classes whose elements have order $p^r$ for some $r \geq 0$ are labelled by pairs of partitions such that the total number is $n$. Call those positive and negative cycles. If $p \neq 2$ the partition is entirely in positive cycles and has to be into $p^r$ parts. If $p = 2$ then there are both positive and negative cycles each with $2^r$ parts for possibly different $r$'s.\\

\begin{eqnarray*}
\tilde{r}( 2 , W_{B_n} ) &=& \sum_{m=0}^n \tilde{r} ( 2, S_m ) \tilde{r} ( 2, S_{n-m})\\
\tilde{r}( p , W_{B_n} ) &=& \tilde{r} (p , S_n )\\
\tilde{OGF} ( 2,B , x ) &=& \tilde{OGF}(2,A,x) \tilde{OGF}(2,A,x) \\
OGF ( p \neq 2, B , x ) &=& OGF(p \neq 2,A,x)
\end{eqnarray*}
\end{proposition}

\begin{proof}
For $B(\mathbb{Z}_2 \wr S_n)$, conjugacy classes are labelled by pairs of partitions \cite{DelducFeher}. Take the prime factorization of the order of the associated element. The negative cycles have a factor of $2$ from the sign flip around the cycle. A negative cycle $(12\cdots n)$ has order $2n$. This ensures that whenever $p \neq 2$, there can be no negative cycles. This reduces to type $A$. If $p=2$, then we are simply taking the least common multiple of a bunch of powers of $2$, so we just need to ensure the parts have length $2^r$.
\end{proof}

\subsection{Type D}

\begin{theorem}
\begin{eqnarray*}
\tilde{r}( 2 , W_{D_n} ) &=& \sum_{m=0}^n \tilde{r} ( 2, S_m ) \tilde{r} ( 2, S_{n-m}, even parts) + \tilde{r}(2,S_n,even lengths)\\
\tilde{r}( p , W_{D_n} ) &=& \tilde{r} ( p ,S_n ) + \tilde{r}(p,S_n,even lengths)\\
\tilde{OGF} ( 2, D  , x ) &=& \tilde{OGF}(2,A,x) (\frac{1}{2} (G(x,1)+G(x,-1))) + \tilde{OGF}(2,A,x^2)\\
&=& \frac{1}{2} \tilde{OGF}(2,A,x)^2 + \frac{1}{2} \tilde{OGF}(2,A,x) (1-x) + \tilde{OGF}(2,A,x^2)\\
\tilde{OGF} (  p \neq 2 , D, x ) &=& \tilde{OGF} (p\neq 2 , A,x) \\
OGF(p \neq 2,D, x) &=& OGF (p\neq 2 , A,x)\\
\end{eqnarray*}
\end{theorem}

\begin{proof}

For $W_{D_n}$ the second partition must have an even number of parts. Also when all cycles are positive of even length, the same partition gives 2 separate conjugacy classes. For the positive cycles we get order as the least common multiple of the lengths. The negative partitions are different only in the first cycle as $(123\cdots n) \to (-n,-1,2,3\cdots n-1)$. $g^n (1\cdots n)= (12\cdots n)$ so has a power of $2$ order if and only if the underlying partition did.

$\tilde{r}(2,S_{n-m},evenparts)$ counts partitions of $n-m$ into powers of 2 length parts and there must be an even number of them. Keep track of the number of parts by $u$; so we only wish to take the sum of only even powers of $u$.

\begin{eqnarray*}
G(x,u) &=& \prod_{j \geq 0} \frac{1}{1 - u x^{2^j}}\\
\sum \tilde{r}(2,S_n,evenparts) x^n &=& \frac{1}{2} (G(x,1) + G(x,-1))\\
G(x,-1) &=& 1-x\\
\sum \tilde{r}(2,S_n,evenparts) x^n &=& \frac{1}{2} G(x,1) + \frac{1}{2} ( 1-x)\\
\end{eqnarray*}

That means that $\tilde{r}(2,S_n,evenparts) = \frac{1}{2} \tilde{r}(2,S_n)$ for all $n \geq 2$.

$\tilde{r}(2,S_n,evenlengths)$ counts partitions of $n$ in powers of $2$ but none can be $2^0$.  If $n$ is odd, then it is $0$. This means we can halve all the parts and remove the condition of even lengths.

\begin{eqnarray*}
\tilde{r}(2,S_n,evenlengths) &=& \tilde{r}(2,S_{n/2})\\
\sum \tilde{r}(2,S_n,evenlengths) x^n &=& \tilde{OGF} (2,A,x^2)\\
\end{eqnarray*}

In contrast, $\tilde{r}(p \neq 2,S_n,even lengths) = 0$ because then $p^r$ are all odd.

\end{proof}

\subsection{Exceptionals}

\begin{lemma}[Exceptional Examples]

The exceptional classes finish off the possible Weyl groups and they can be read from \cite{Carter,Zhang}.

\begin{itemize}
\setlength\itemsep{-1em}
\item $W_{D_4}$ $r(2)= 10 \; r(3) = 1$ \footnote{This isn't exceptional in some ways but very exceptional in others.}\\
\item $W_{F_4}$ $r(2) = 13\; r(3) = 3$\\
\item $W_{G_2}$ $r(2) = 3\; r(3) = 1$\\
\item $W_{E_6}$ $r(2)=9 \; r(3)=4 \; r(5)=1$\\
\item $W_{E_7}$ $r(2)=23 \; r(3)=4 \; r(5)=1 \; r(7)=1$\\
\item $W_{E_8}$ $r(2) = 31\;r(3)=6\;r(5)=2\;r(7)=1$\\
\end{itemize}

\end{lemma}

\begin{lemma}[$H_{3/4}$]

If we relax the condition to finite Coxeter groups from Weyl groups, the symmetries of the dodecahedron and 600-cell are also allowed.

\begin{itemize}
\setlength\itemsep{-1em}
\item $H_3$ $r(2)=3$ $r(3)=1$ $r(5)=2$\\
\item $H_4$ $r(2)=6$ $r(3)=2$ $r(5)=5$\\
\end{itemize}
\end{lemma}

\begin{proof}
This can be read off from Sage:

\begin{verbatim}
W = ReflectionGroup(['H',3]); W
CW=W.conjugacy_classes_representatives();
orders=[CW[i].order() for i in range(0, len(CW))]; orders
\end{verbatim}
\end{proof}

\section{Analogy with Hilbert Schemes}

\begin{definition}[$Hilb^n \mathbb{C}^2$]
We may resolve $\mathbb{C}^{2n}/S_n$ by taking ideals of length $n$ in $\mathbb{C}[x,y]$. Maximal ideals like $(x-a,y-b)$ will be points and if there are $n$ disjoint points we get an ideal of length $n$. Similarly define $Hilb^n$ for other surfaces resolving $\mathbb{C}^2 / G$.
\end{definition}

If we didn't work algebraically, but topologically instead we would see a contractible space quotiented by $S_n$. The action isn't free which weakens the analogy. Repeat the same process for $\mathbb{C}^2 / G$ with $G$ a binary Platonic group indexed by an ADE Lie algebra. In the purely homotopic world this resembles $B(G \wr S_n)$ as a proxy for $Hilb^n (\mathbb{C}^2/G)$. We cannot accomodate $T$ equivariance because that escapes the world of quotients by finite groups. It can only be approximated by $\mathbb{Z}_{m_1} \times \mathbb{Z}_{m_2}$ for which one may refer to \eqref{cor:TwoCyclics}.

Nakajima takes $K_T ( Hilb_n (\mathbb{C}^2/G))$ in a procedure that gives the $U_q ( \hat{\mathfrak{g}})$ for the corresponding Lie algebra \cite{Nakajima}. For $G=e$ the trivial group, this constructs a q-boson algebra. $G=\mathbb{Z}_2$ like the type B/C Weyl groups above corresponds to $A_1$ or $U_q (\hat{\mathfrak{sl}_2})$.

\if\doAlgSec1
\begin{definition}[``Algebra" Structure]
Let $G_\bullet$ be a sequence of finite groups satisfying $\forall n \; m \; \; G_n \times G_m \hookrightarrow G_{n+m}$. In this case define a graded ring by

\begin{eqnarray*}
\mathcal{A} &=& \bigoplus_n K^0 ( BG_n ) [n]
\end{eqnarray*}

$g_1 \in G_n$ and $g_2 \in G_m$ defines a group element $g_1 \times g_2 \in G_{n+m}$. Conjugating each of $g_1$ and $g_2$ does not affect the conjugacy class in $G_{n+m}$. When the order of $g_{1,2}$ are $p^r$ for some $r$, then $g_1 \times g_2$ will too. That means it induces the map $K^0 (B G_n )[n] \times K^0 (BG_m ) [m] \to K^0 (BG_{n+m})[n+m]$. This gives the graded ring structure. This ``algebra" has distinguished sub``algebras" for every prime and every subset of the primes by only keeping the allowed conjugacy classes.
\end{definition}
\fi

\begin{lemma}
Conjugacy classes in $G \wr S_n$ are given by a partition of $n$ and a labelling of each part with a conjugacy class of $G$. We can decorate all with the same conjugacy class of $G$ and get an analog of the decomposition into q-boson algebras.
\end{lemma}

\begin{cor}
\begin{eqnarray*}
\tilde{OGF}(p,G \wr A,x) &=& \tilde{OGF}(p,A,x)^{\tilde{r}(p,G)}\\
\end{eqnarray*}

This allows us to recover $K^0 (B (G \wr S_n))$ by looking at the appropriate coefficient of these functions for each $p$. We can also easily let $G$ vary if need be.

\end{cor}

\begin{proof}
First we give a partition of $n$ and to have $p$ power order it must be into $p$ power parts. The conjugacy classes coloring the parts also have to be $p$ power to maintain this condition. That gives $\tilde{r}(p,G)$ chocies for the colors of the parts. This procedure is in bijection with counting the number of ways of dividing $n$ up into $\tilde{r}(p,G)$ parts and then each one of those gets the structure of a partition. The values of $\tilde{r}(p,G)$ of the binary Platonic groups were already given in \cref{Platonic}. Together on the generating function this amounts to taking the $\tilde{r}(p,G)$'th power.
\end{proof}

\begin{lemma}

\begin{eqnarray*}
F(p, G \wr A , s) &\equiv& \mathcal{M} ( \log \tilde{OGF}(p,G \wr A,e^{-t}))\\
F(p, G \wr A , s) &=& \tilde{r}(p,G) \Gamma(s)\frac{\zeta(s+1)}{1-p^{-s}}\\
\sum y^m F(p, G_m \wr A , s) &=& \tilde{OGF}(p,G_\bullet , y) \Gamma(s)\frac{\zeta(s+1)}{1-p^{-s}}\\
\end{eqnarray*}

\end{lemma}

\begin{proof}

\begin{eqnarray*}
F(p, G \wr A, s) &\equiv& \mathcal{M} ( \log \tilde{OGF}(p,A,e^{-t})^{\tilde{r}(p,G)})\\
&=& \tilde{r}(p,G) \mathcal{M} ( \log \tilde{OGF}(p,A,e^{-t})) = \tilde{r}(p,G) F(p,A,s)\\
F(p, A,s) &=& \Gamma(s)\sum_{j\ge 0, k \ge 1} \frac{(kp^j)^{-s}}{k} = \Gamma(s)\frac{\zeta(s+1)}{1-p^{-s}}\\
F(p, G \wr A ,s) &=& \tilde{r}(p,G) \Gamma(s)\frac{\zeta(s+1)}{1-p^{-s}}\\
\end{eqnarray*}

\end{proof}

\section{Conclusion}

\par We have applied the theorem of L{\"u}ck to compute $K^0 (BG)$ for cases of finite groups that are relevant to Platonic solids, Weyl groups and Hilbert schemes. When the groups come in a natural countable family, we may form generating functions akin to $\eta (q)$ and partitions.\if\doAlgSec1 There is also a similar algebra structure to Nakajima's construction of quantum affine algebras.\fi

These were example computations that leave many questions raised. These are driven by understanding the relations between the algebraic and topological K-theories of $\mathbb{C}^2/G$. This goes into what is lost and what is kept by the comparison map \cite{Gillet}. There is also the distinction between genuine equivariant K theory and the Borel equivariant K theory that we have considered here. In addition, conjugacy classes in Weyl groups are related to nilpotent coadjoint orbits/W-algebras via \cite{KazhdanLusztig,Sevostyanov} so we can translate the prime power conditions there too. 

\bibliographystyle{ieeetr}
\bibliography{DiscreteGroupKTheory}

\begin{thebibliography}{10}

\bibitem{AtiyahSegal}
M.~F. Atiyah and G.~B. Segal, ``Equivariant $k$-theory and completion,'' {\em
  J. Differential Geom.}, vol.~3, no.~1-2, pp.~1--18, 1969.

\bibitem{Luck}
W.~L{\"u}ck, ``Rational computations of the topological k-theory of classifying
  spaces of discrete groups,'' {\em Journal f{\"u}r die reine und angewandte
  Mathematik (Crelles Journal)}, vol.~2007, no.~611, pp.~163--187, 2007.

\bibitem{Tomoda}
S.~Tomoda and P.~Zvengrowski, ``Remarks on the cohomology of finite fundamental
  groups of 3-manifolds,'' {\em Geometry \& Topology Monographs}, vol.~14,
  pp.~519--556, 2008.

\bibitem{OEIS006519}
N.~Sloane, ``The on-line encyclopedia of integer sequences.''
\newblock Sequence A006519.

\bibitem{OEIS018819}
N.~Sloane, ``The on-line encyclopedia of integer sequences.''
\newblock Sequence A018819.

\bibitem{OEIS062051}
N.~Sloane, ``The on-line encyclopedia of integer sequences.''
\newblock Sequence A062051.

\bibitem{Latapy}
M.~Latapy, ``Partitions of an integer into powers.,'' in {\em DM-CCG},
  pp.~215--228, 2001.
\newblock Available at
  \url{https://www-complexnetworks.lip6.fr/~latapy/Publis/dmccg01.pdf}.

\bibitem{DelducFeher}
F.~Delduc and L.~Feher, ``{Conjugacy classes in the Weyl group admitting a
  regular eigenvector and integrable hierarchies},'' {\em J. Phys.}, vol.~A28,
  pp.~5843--5882, 1995.
\newblock Available at
  \url{http://www.iaea.org/inis/collection/NCLCollectionStore/_Public/26/072/26072147.pdf}.

\bibitem{Carter}
R.~W. Carter, ``Conjugacy classes in the weyl group,'' {\em Compositio
  Mathematica}, vol.~25, no.~1, pp.~1--59, 1972.

\bibitem{Zhang}
S.~Zhang, Y.-Z. Zhang, P.~Wang, J.~Cheng, and H.~Yang, ``On pointed hopf
  algebras with weyl groups of exceptional type,'' {\em arXiv preprint
  arXiv:0804.2602}, 2008.

\bibitem{Nakajima}
H.~Nakajima, ``Quiver varieties and finite dimensional representations of
  quantum affine algebras,'' {\em Journal of the American Mathematical
  Society}, vol.~14, no.~1, pp.~145--238, 2001.

\bibitem{Gillet}
H.~Gillet, ``Comparing algebraic and topological k-theory,'' {\em Lect. Notes
  in Math}, vol.~1491, pp.~55--99, 1992.

\bibitem{KazhdanLusztig}
D.~Kazhdan and G.~Lusztig, ``Fixed point varieties on affine flag manifolds,''
  {\em Israel Journal of Mathematics}, vol.~62, no.~2, pp.~129--168, 1988.

\bibitem{Sevostyanov}
A.~Sevostyanov, ``Conjugacy classes in weyl groups and qw algebras,'' {\em
  Advances in Mathematics}, vol.~228, no.~3, pp.~1315--1376, 2011.

\end{thebibliography}

\end{document}